\documentclass[11pt]{article}

\usepackage{amsmath, amsthm}
\usepackage{amssymb}
\usepackage[alphabetic]{amsrefs}
\usepackage{hyperref}
\usepackage[shortlabels]{enumitem}
\usepackage[margin=1.5in]{geometry}

\newtheorem{theorem}{Theorem}[section]
\newtheorem{lemma}[theorem]{Lemma}
\newtheorem{corollary}[theorem]{Corollary}
\newtheorem{proposition}[theorem]{Proposition}

\theoremstyle{remark}
\newtheorem{remark}[theorem]{Remark}

\newtheorem{example}[theorem]{Example}
\newtheorem{question}[theorem]{Question}

\numberwithin{equation}{section}

\newcommand{\Cstar}{\mathrm{C}^*}
\newcommand{\N}{\mathbb{N}}

\newcommand{\C}{\mathbb{C}}
\newcommand{\Id}{\mathrm{Id}}
\DeclareMathOperator{\Lin}{span}

\title{On the Lie ideals of $\Cstar$-algebras}
\author{Leonel Robert}
\date{}

\begin{document}
\maketitle


\begin{abstract} 
	Various questions on Lie ideals of $\Cstar$-algebras are investigated. They fall  roughly under the  following topics:  relation of Lie ideals to closed two-sided ideals;  
	Lie ideals spanned by special classes of elements such as commutators, nilpotents, and the range of polynomials; characterization of Lie ideals as similarity invariant subspaces.
	\end{abstract}
	

\section*{Introduction}
This paper deals with Lie ideals in $\Cstar$-algebras.
Like other investigations on this topic (\cites{miers, BKS}), we use, and take inspiration from,  Herstein's work  on the  Lie ideals of   semiprime rings. The abundance of semiprime ideals in a $\Cstar$-algebra -- e.g., the norm-closed ideals -- plus a number of  $\Cstar$-algebra techniques -- approximate units, polar decompositions, functional calculus -- make it possible to further develop the results of the purely algebraic setting in the $\Cstar$-algebraic setting.  

The contributions in the present paper, though  varied,  revolve around the following themes:  the commutator equivalence
of Lie ideals to two-sided ideals; the study of Lie ideals generated by special elements such as nilpotents and projections  and by the range of polynomials; the characterization of Lie ideals as subspaces  invariant by  similarities.
These topics  have been studied before,  and  this paper is a direct beneficiary of works such as \cite{marcoux09}, \cite{BKS}, and \cite{bresar-klep1}.

A  selection of results in this paper follows:
Let $A$ be a $\Cstar$-algebra. We show below that the following are true:
\begin{enumerate}[(i)]
	\item
	The closed two-sided ideal generated by the commutators of $A$ is also  the $\Cstar$-algebra generated by the commutators of $A$. (Theorem \ref{commutator})
	\item
	The closure of the linear span of the square zero elements agrees with the closure of the linear span of the commutators. If $A$ is unital and
	without 1-dimensional representations, then the linear span of the square zero elements agrees with  the linear span of the commutators. (Corollary \ref{linN2} and Theorem \ref{linN2comm}.)
	
	\item
	If $A$ is unital and has no bounded traces and $f$ is a nonconstant polynomial in noncommuting variables with coefficients in $\C$, then there exists $N$ such that every element	of $A$ is a linear combination of at most $N$ values of $f$ on $A$. If  $f(\C)=\{0\}$ (e.g., $f(x,y)=[x,y]$), then there exist $\Cstar$-algebras where the least such $N$ can be arbitrarily large.
	(Corollary \ref{polypop} and Example \ref{counterpop}.)
	
	\item
	If $A$ is unital and either simple, or without bounded traces, or a von Neumann algebra, then a subspace  $U$ of $A$ is a Lie ideal of $A$ if and only if 
	$(1+x)U(1-x)\subseteq U$ for all  square zero  elements $x$ in $A$. (Corollaries \ref{simneu} and \ref{simpleamitsur}.)
\end{enumerate}	

\section{From pure algebra to  $\Cstar$-algebras}
Let us fix some notation: 
\begin{center}
	\emph{Throughout the paper $A$ denotes a $\Cstar$-algebra.} 
\end{center}
Let $x$ and $y$ be elements in $A$. Then $[x,y]$ denotes the element $xy-yx$ (the commutator of $x$ and $y$). Let $X$ and $Y$ be subsets of $A$.  Then   $X+Y$,
$XY$, and $[X,Y]$ denote the linear spans of the elements of the form $x+y$, $xy$, and $[x,y]$, with $x\in X$ and $y\in Y$, respectively. The  linear span of $X$ is denoted by $\Lin(X)$. The $\Cstar$-algebra  and the closed two-sided ideal generated by $X$  are denoted by $\Cstar(X)$ and  $\Id(X)$, respectively.    (For the 2-sided ideal algebraically generated by $X$  we simply write  $A XA$.)
From the identity $[xy,a]=[x,ya]+[y,ax]$, used inductively, we deduce that \begin{equation}\label{linearizing}
[X^n,A]\subseteq [X,A]
\end{equation}
for any set $X\subseteq A$ and all $n\in \N$. We sometimes  refer to this fact as the ``linearizing property of $[\cdot,A]$".

A subspace $L$ of $A$ is called a Lie ideal if  it satisfies that $[L,A]\subseteq L$. 
We will make frequent  use of the following elementary  lemma:
\begin{lemma}\label{elementary}
	Let $L$ be a Lie ideal of $A$.  Then
	$A[L,L]A\subseteq L+L^2$.
\end{lemma}
\begin{proof}
	We have $[[L,L],A]\subseteq [L,L]$, by Jacobi's identity. Thus,
	\begin{align*}
	[L,L]A &\subseteq A[L,L]+[[L,L],A]\\
	&\subseteq A[L,L]+[L,L],
	\end{align*}
	Multiplying by $A$ on the left we get   $A[L,L]A\subseteq A[L,L]$. Finally, from the identity $a[l_1,l_2]=[al_1,l_2]-[a,l_2]l_1$
	we deduce that $A[L,L]\subseteq L+L^2$, as desired.
\end{proof}

The following theorem of Herstein is the basis of many of our arguments in this section (it holds for semiprime rings without 2-torsion): 
\begin{theorem}[\cite{herstein2}*{Theorem 1}]
	Let $L$ be a Lie ideal of $A$. Then
	$[t,[t,L]]=0$ implies $[t,L]=0$ for all $t\in A$.
\end{theorem}

Combining   Herstein's theorem and Lemma  \ref{elementary} we get the following theorem:

\begin{theorem}\label{commutator}
	The closed two-sided ideal generated by $[A,A]$ agrees with the $\Cstar$-algebra generated by $[A,A]$. In fact, 
	$\Id([A,A])=\overline{[A,A]+[A,A]^2}$.
\end{theorem}	
\begin{proof}
	Let $I=\Id([ [A,A] , [A,A] ])$. Then  $[[x,y],[[x,y],A/I]]=0$ for all $x,y\in A/I$. Herstein's theorem implies that   $[[x,y],A/I]=0$ for all $x,y\in A/I$. That is, $[[A/I,A/I],A/I]=0$.
	Herstein's theorem again implies that $[A/I,A/I]=0$; i.e., $[A,A]\subseteq I$. 
	On the other hand, $I\subseteq \overline{[A,A]+[A,A]^2}$, by  Lemma \ref{elementary}. So, 
	\[
	\Id([A,A])\subseteq I\subseteq \overline{[A,A]+[A,A]^2}\subseteq \Cstar([A,A]).
	\]   
	Since $\Cstar([A,A])\subseteq \Id([A,A])$, these inclusions must be equalities.
\end{proof}

The following lemma   is easily derived  from the existence of approximately central approximate units for the
closed two-sided ideals of $A$:
\begin{lemma}[\cite{miers}*{Lemma 1}, \cite{BKS}*{Proposition 5.25}]\label{bunce}
	Let $I$  be a  closed two-sided ideal of $A$. Then
	\[
	\overline{[I,I]}=\overline{[I,A]}=I\cap \overline{[A,A]}.
	\]
\end{lemma}

Bre\v{s}ar, Kissin, and Shulman show in \cite[Theorem 5.27]{BKS} that $\overline{[L,A]}=\overline{[\Id([L,A]),A]}$ for any Lie ideal $L$ of $A$. In the theorem below  we give a short proof of this important theorem:
\begin{theorem}\label{bresar}
	Let $L$ be  a Lie ideal of $A$. Then 
	\begin{enumerate}[(i)]
		\item
		$\Id([L,A])=\overline{[L,A]+[L,A]^2}$.
		\item
		$\overline{[\Id([L,A]),A]}=\overline{[L,A]}=\overline{[[L,A],A]}$.
	\end{enumerate}	
\end{theorem}
\begin{proof}
	(i) We follow a line of argument similar to the proof of  Theorem \ref{commutator}. Let $M=[L,A]$ and $I=\Id([M,M])$. Let $\widetilde L$ and $\widetilde M$ denote the images of $L$ and $M$ in $A/I$ by the quotient map. 
	Then  $[\widetilde M,[\widetilde M,A/I]]=0$. By Herstein's theorem, $[\widetilde M,A/I]=0$; i.e., $[[\widetilde L,A/I],A/I]$. By Herstein's theorem again,   $[\widetilde L,A/I]=0$; i.e., $[L,A]\subseteq I$. On the other hand, $I\subseteq \overline{M+M^2}=\overline{[L,A]+[L,A]^2}$, by Lemma \ref{elementary}. So, 
	\[
	\Id([L,A])\subseteq I\subseteq \overline{[L,A]+[L,A]^2}\subseteq \Cstar([L,A]).
	\]	
	Since $\Cstar([L,A])\subseteq \Id([L,A])$, all these inclusions must be equalities.
	
	(ii) By (i) and the linearizing property of $[\cdot,A]$ recalled in \eqref{linearizing}, we have that
	\[
	[\Id([L,A]),A]=[\overline{[L,A]+[L,A]^2},A]\subseteq \overline{[[L,A],A]}.
	\]
	Thus,
	$\overline{[\Id([L,A]),A]}\subseteq \overline{[[L,A],A]}\subseteq\overline{[L,A]}$. 
	On the other hand, 
	\[
	[L,A]\subseteq \Id([L,A])\cap [A,A] \subseteq \overline{[\Id([L,A]),A]},
	\]
	(the second inclusion  by Lemma \ref{bunce}). This completes the proof.
\end{proof}	

\begin{lemma}\label{basiclemma}
	Let $L$ be a  closed Lie ideal of $A$ such that  $\Id(L)=\Id([L,A])$ and $L\subseteq \overline{[A,A]}$. Then $L=\overline{[\Id(L),A]}$. 
\end{lemma}
\begin{proof}
The inclusion  $L\subseteq \overline{[\Id(L),A]}$ follows from $L\subseteq \overline{[A,A]}\cap \Id(L)$ and Lemma \ref{bunce}. As for the opposite inclusion,  we have $\overline{[\Id(L),A]}=\overline{[\Id([L,A]),A]}$, by assumption, and  $\overline{[\Id([L,A]),A]}\subseteq L$, by Theorem \ref{bresar}.
\end{proof}

The following is an improvement on Theorem \ref{bresar} (ii) obtained by the same technique:
\begin{theorem}\label{betterbresar}
	Let $K$ and $L$ be Lie ideals of $A$. Then $\overline{[K,L]}=\overline{[\Id([K,L]),A]}$.
\end{theorem}
\begin{proof}
	Let $M=[K,L]$. Notice that $M$ is again a Lie ideal (by Jacobi's identity). We will deduce that $\overline{M}=\overline{[\Id(M),A]}$  from the previous lemma. We clearly have that $\overline{M}\subseteq \overline{[A,A]}$. Let $I=\Id([M,A])$ and let $\widetilde K$, $\widetilde L$, and $\widetilde M$ denote the images of $K$, $L$, and $M$ in the quotient by this ideal. From  $[\widetilde M,A/I]=0$ and $[\widetilde K,\widetilde L]= \widetilde M$  we get that 
	$[[\widetilde K,\widetilde L],\widetilde L]=0$. By Herstein's theorem, 
	$[\widetilde K,\widetilde L]=0$; i.e, 
	$M=[K,L]\subseteq I$. It follows that $\Id(M)=\Id([M,A])$. By Lemma \ref{basiclemma}, $M=\overline{[\Id(M),A]}$, as desired.
\end{proof}

\begin{remark}
	The arguments in  Theorems \ref{commutator}, \ref{bresar}, and \ref{betterbresar}  rely crucially on the fact that the closed two-ideals of a $\Cstar$-algebra are semiprime. This makes it possible to apply Herstein's theorem in the quotient by a closed two-sided ideal. Turning to non-closed Lie ideals, if we impose the semiprimeness of a suitable non-closed two-sided ideal at the outset, part of those same arguments still goes through. We may obtain in this way,  for instance, the following result:
	\emph{If $L$ is a Lie ideal of $A$ such that the two-sided ideal generated by $[[L,A],[L,A]]$ is semiprime then  (i) $A[L,A]A=[L,A]+[L,A]^2$, and (ii)
		$[A[L,A]A],A]=[[L,A],A]$}. 
	To get (i) we  proceed as in Theorem \ref{bresar} (i):  Setting $M=[L,A]$ and  $I=A[M,M]A$ and applying Herstein's theorem in $A/I$ in much the same way as we did in Theorem \ref{bresar} (i) we arrive at  $[L,A]\subseteq I$.  We then have the inclusions $A[L,A]A\subseteq I\subseteq [L,A]+[L,A]^2$, which must in fact be   equalities.
	To get (ii) we apply (i) and the linearizing property of $[\cdot,A]$:
	\[
	[A[L,A]A,A]=[[L,A]+[L,A]^2,A]=[[L,A],A].
	\] 
\end{remark}

Next we discuss another variation on Theorem \ref{bresar} for non-closed Lie ideals. This time we make use of the Pedersen ideal. Recall that the Pedersen ideal of a $\Cstar$-algebra is the smallest dense two-sided ideal of the algebra (see \cite[5.6]{pedersen}).
Given a $\Cstar$-algebra $B$, we denote its Pedersen ideal by  $\mathrm{Ped}(B)$.

\begin{lemma}\label{PPPA}
	Let $I$ be a closed two-sided ideal of $A$. Then 
	\[
	[\mathrm{Ped}(I),\mathrm{Ped}(I)]=[\mathrm{Ped}(I),A].
	\]	
\end{lemma}	
\begin{proof}
	Let $P=\mathrm{Ped}(I)$.	The subspace $P^2$ is a dense two-sided ideal of $I$. Since $P$ is the minimum such 
	ideal, we must have that $P=P^2$. From  $[P,A]=[P^2,A]$  and   the identity $[xy,a]=[x,ya]+[y,ax]$  we get that $[P^2,A]\subseteq [P,P]$.  	
\end{proof}

\begin{theorem}\label{pedersen}
	Let $L$ be a Lie ideal of $A$ and let $P=\mathrm{Ped}(\Id([L,A]))$. Then
	\[
	[P,P]=[L,P]=[[L,A],P].
	\] 
	Furthermore, if $L\subseteq P$ then $[L,A]=[P,P]$.	
\end{theorem}	
\begin{proof}
	In the course of proving Theorem \ref{bresar}  we have shown  that 
	$\Id([L,A])=\Id([[L,A],[L,A]])$. Therefore, the two-sided ideal $A[[L,A],[L,A]]A$ is   dense in $\Id([L,A])$. Since $P$ is the smallest such ideal,  $P\subseteq A[[L,A],[L,A]]A$. Hence,
	\[
	[P,P]\subseteq [A[[L,A],[L,A]]A,P]\subseteq [[L,A]+[L,A]^2,P]\subseteq [[L,A],P]\subseteq [L,P].
	\]
	But $[L,P]\subseteq [P,P]$, by Lemma \ref{PPPA} . Thus, the inclusions above must be equalities.
	
	Suppose now that $L\subseteq P$. Then 
	$[L,P]\subseteq [L,A]\subseteq [P,A]=[P,P]$,
	the latter equality by Lemma \ref{PPPA}. Since $[L,P]=[P,P]$, these inclusions  must be equalities.	
\end{proof}	

\begin{corollary}\label{smallest}
	Among the Lie ideals $L$ such that $\overline{[L,A]}=\overline{[A,A]}$, the Lie ideal 
	\[
	[\mathrm{Ped}(\Id([A,A])),\mathrm{Ped}(\Id([A,A]))]
	\] is the smallest.	
\end{corollary}	
\begin{proof}
	Let $P=\mathrm{Ped}(\Id([A,A]))$. Then
	\begin{align*}
	\overline{[[P,P],A]} &=\overline{[[\Id([A,A]),\Id([A,A])],A]}\\
	&=\overline{[[\Id([A,A]),A],A]}\\
	&=\overline{[\Id([A,A]),A]}\\
	&=\overline{[A,A]}.
	\end{align*}
	The second  equality holds by Lemma \ref{bunce} and the third and fourth by Theorem \ref{bresar}. 
	Thus, $[P,P]$ is a Lie ideal satisfying that $\overline{[L,A]}=\overline{[A,A]}$.
	
	Suppose now that $L$ is a Lie ideal such that $\overline{[L,A]}=\overline{[A,A]}$.
	By Theorem \ref{pedersen}, $[P,P]=[L,P]\subseteq L$. So $L$ contains $[P,P]$.
\end{proof}

It seems possible that under some $\Cstar$-algebra regularity condition, such as $A$ being pure (i.e, having almost unperforated and almost divisible Cuntz semigroup), it is the case that for every Lie ideal $L$ there exists  a  two-sided  -- possible non-closed -- ideal $I$ such that $[L,A]=[I,A]$ (in the language of \cite{BKS}, $L$ and $I$ are called commutator equal). At present, we don't even have an answer to the following question:
\begin{question}
	Is there a $\Cstar$-algebra $A$ and a Lie ideal $L$ of $A$, such that $[L,A]\neq [I,A]$
	for all two-sided (possibly non-closed) ideals $I$ of $A$?	
\end{question}

We turn now to Lie ideals of $[A,A]$. 
A linear subspace  $U\subseteq A$ is called a Lie ideal of $[A,A]$ if 
$[U,[A,A]]\subseteq U$. 
Herstein's \cite[Theorem 1.12]{herstein} implies that if $A$ is   simple  and unital then a Lie ideal of $[A,A]$ is automatically a Lie ideal of $A$ (this holds for simple 
rings without 2-torsion).  In Theorem \ref{LieAA} below we show that the simplicity assumption can be dropped  for \emph{closed} Lie ideals of $[A,A]$. The key of the argument is again to apply a theorem of Herstein (Lemma \ref{algebraUU} below) in the quotient by a suitable closed two-sided ideal.

\begin{lemma}\label{4U}
	Let $U$ be a  Lie ideal of $[A,A]$. Let $V=[U,U]$, $W=[V,V]$, and $X=[W,W]$.
	Then $A[X,X]A\subseteq [U,U]+[U,U]^2$.
\end{lemma}
\begin{proof}
	(Cf. \cite[Lemma 1.7]{herstein}.)
	In the following inclusions  we make use of Jacobi's identity and the fact that $U$ is a Lie ideal of $[A,A]$:
	\begin{align*}
	[[U,U],A] &\subseteq [U,[A,A]]\subseteq U,\\
	[[U,U],[A,A]]&\subseteq [[U,[A,A]],U]\subseteq [U,U].
	\end{align*}
	That is, $[V,A]\subseteq U$ and $V$ is a Lie ideal of $[A,A]$. We deduce similarly that  $[A,W]\subseteq V$ and that  $W$  and $X$ are Lie ideals of $[A,A]$. 
	Finally, since $V\subseteq [A,A]$ we have that $[V,V]\subseteq V$; i.e., $W\subseteq V$. We deduce  similarly that
	$[X,X]\subseteq X$. Having made this preparatory remarks, we attack the lemma:
	\begin{align*}
	[X,X]A &\subseteq A[X,X]+[[X,X],A]\\
	&\subseteq A[X,X]+X\\
	&\subseteq AX+X.
	\end{align*}
	Hence, $A[X,X]A\subseteq  AX=A[W,W]$. Using now that $a[w_1,w_2]=[aw_1,w_2]-[a,w_2]w_1$ we get that
	\begin{align*}
	A[W,W] &\subseteq [A,W]+[A,W]W\\
	&\subseteq V+VW\\
	&\subseteq V+V^2.
	\end{align*}
	Thus, $A[X,X]A\subseteq V+V^2$, as desired.
\end{proof}

\begin{lemma}\label{algebraUU}
	Let $U$ be a  Lie ideal of $[A,A]$.
	If $[[U,U],A]=0$ then $[U,A]=0$.
\end{lemma}	
\begin{proof}
	See \cite[Theorem 1.11]{herstein} for the case of simple rings without 2-torsion. See  \cite[Exercise 17, page 344]{rowen} for the extension to  semiprime rings  without 2-torsion (e.g.,  $\Cstar$-algebras).
\end{proof}

\begin{theorem}\label{LieAA}
	A (norm) closed Lie ideal of $[A,A]$  is a  Lie ideal of $A$.
\end{theorem}	

\begin{proof}
	Let $U$ be a closed Lie ideal of $[A,A]$. Consider the sets  
	$V=[U,U]$, $W=[V,V]$ and  $X=[W,W]$. 
	Let $I=\Id([X,X])$. Let $\widetilde U$ denote the image of $U$ in $A/I$ by the quotient map. Define  $\widetilde V$, $\widetilde W$, and $\widetilde X$ similarly.
	Then $[\widetilde X,\widetilde X]=0$, which, by  Lemma \ref{algebraUU},  implies that
	$[\widetilde X,A/I]=0$. That is, $[[\widetilde W,\widetilde W], A/I]=0$. Again by Lemma \ref{algebraUU} we get that $[\widetilde W, A/I]=0$. That is, $[[\widetilde V,\widetilde V],A/I]=0$. Two more applications of Lemma \ref{algebraUU} then yield that  $[\widetilde U, A/I]=0$. That is, $[U,A]\subseteq I$. Hence,
	\[
	\Id([U,A])\subseteq I\subseteq \overline{[U,U]+[U,U]^2}\subseteq \Id([U,U]).
	\]
	In the second inclusion we have used Lemma \ref{4U}.
	Since $\Id([U,U])\subseteq \Id([U,A])$, all these must be equalities.
	Taking commutators with $A$ and using \eqref{linearizing}
	we get 
	\[	\overline{[\Id([U,A]),A]}=\overline{[[U,U]+[U,U]^2,A]}=\overline{[[U,U],A]}\subseteq U.	
	\]
	Lemma \ref{bunce}, on the other hand, implies that
	\[
	[U,A] \subseteq \Id([U,A])\cap \overline{[A,A]}=\overline{[\Id([U,A]),A]}.
	\]
	Hence, $[U,A]\subseteq U$; i.e., $U$ is  a Lie ideal of $A$.
\end{proof}

\section{Nilpotents and polynomials}
In this section we look at closed Lie ideals spanned by nilpotents and by the range of polynomials.

For each natural number $k\geq 2$ let $N_k$ denote the set of nilpotent elements of $A$ of order exactly $k$. Since the set $N_k$ is invariant by unitary conjugation (and by similarity), the closed subspace $\overline{\Lin(N_k)}$ is a Lie ideal of $A$ (see \cite{pedersen} and Theorem \ref{closedinvariance} below).

The following lemma is surely well known:
\begin{lemma}\label{aluthge}
	Every element of  $N_k$  is a sum of $k-1$ commutators for all $k\geq 2$.  
\end{lemma}
\begin{proof}
	Let $x$ be a nilpotent of order at most $k$ (i.e., in $\bigcup_{j\leq k} N_j$). 
	Let $x=v|x|$ be the polar decomposition of $x$ in $A^{**}$. Let $\tilde x=|x|^{\frac 1 2}v|x|^{\frac 1 2}$ (the Aluthge transform of $x$). Observe that $x=[v|x|^{\frac 1 2},|x|^{\frac 1 2}]+\tilde x$. Also,
	\[
	\tilde x^{k-1}(\tilde x^{k-1})^*=|x|^{\frac 1 2}x^{k-1} v^*(x^{k-2})^*|x|^{\frac 1 2}=0,
	\] 
	where we have used that $|x|^{\frac 1 2}x^{k-1}=0$ (since  $|x|^{\frac 1 2}\in \Cstar(x^*x)$ and $(x^*x)x^{k-1}=0$). 
	Thus $\tilde x$ is a nilpotent of order at most $k-1$. Continuing this process inductively we arrive at the desired result.
\end{proof}	

For each $k\in \N$ let  $I_k$ denote the intersection of  the kernels of all representations of $A$ of dimension at most $k$.  Notice  that $I_1=\Id([A,A])$ and that $I_1\supseteq I_2\supseteq \dots$.  It is not hard to show that  $I_k$ is the smallest closed two-sided ideal the quotient by which is a  $k$-subhomogeneous $\Cstar$-algebra (i.e., one  whose irreducible representations are at most $k$-dimensional).

\begin{theorem}\label{linNk}
	$\overline{\Lin(N_k)}=\overline{[I_{k-1},A]}$ for all $k\geq 2$.	
\end{theorem}	
\begin{proof}
	It is well known that $\Id(N_k)=I_{k-1}$ (e.g., see \cite[Lemma 6.1.3]{ara-mathieu}). We must then show that 
	$\overline{\Lin(N_k)}=\overline{[\Id(N_k),A]}$. Let  $I=\Id([N_k,A])$.
	Let $x\in N_k$. Since  $[x ,A]\subseteq I$,  the quotient map sends
	$x$ to the center of $A/I$. But the center, being a commutative $\Cstar$-algebra, cannot contain nonzero nilpotents. 
	Thus, $x\in I$. This shows that  $N_k\subseteq \Id([N_k,A])$. On the other hand, $N_k\subseteq [A,A]$ by Lemma \ref{aluthge}.
	Thus, $\overline{\Lin (N_k)}=\overline{[\Id(N_k),A]}$ by Lemma \ref{basiclemma}.  
\end{proof}

\begin{corollary}\label{linN2}
	$\overline{\Lin (N_2)}=\overline{[A,A]}$.	
\end{corollary}	
\begin{proof}
	The previous theorem implies that  $\overline{\Lin (N_2)}=\overline{[\Id([A,A]),A]}$. On the other hand,  $\overline{[\Id([A,A]),A]}=\overline{[A,A]}$, by Theorem \ref{bresar} (ii)
	applied with $L=A$.
\end{proof}

The following corollary is merely a restatement of Corollary \ref{linN2}
\begin{corollary}
	A  positive bounded functional on $A$ is a trace if and only if it vanishes on $N_2$. 
\end{corollary}	

\begin{question}\label{linN2question}
	Is $[A,A]=\Lin(N_2)$? Is $\Lin(N_2)$ a Lie ideal?
\end{question}
We will return to these questions in Section \ref{simspan}.

Combining Corollary \ref{linN2} and Theorem \ref{LieAA} of the previous section we can prove the following 
$\Cstar$-algebraic version of a theorem of Amitsur for simple rings (\cite[Theorem 1]{amitsur}):

\begin{theorem}\label{closedinvariance}
	A closed subspace  $U$ of $A$ is a Lie ideal if and only if 
	$(1+x)U(1-x)\subseteq U$ for all $x\in N_2$.
\end{theorem}	 
\begin{proof}
	Say $U$ is a Lie ideal.  Let $u\in U$ and $x\in N_2$. Then 
	\[
	(1+ x)u(1- x)=u+[x,u]+\frac{1}{2}[x,[x,u]]\in U.
	\]
	
	Suppose now that $(1+x)U(1-x)\subseteq U$ for all $x\in N_2$.	Let $u\in U$ and  $x\in N_2$. 
	Then
	\begin{align*}
	[x,u]-xux &=(1+ x)u(1- x)-u\in U,\\
	[x,u]+xux &=-(1- x)u(1+ x)+u\in U.
	\end{align*} 
	Hence $[u,x]\in U$. That is, $[U,N_2]\subseteq U$. Passing to the span of $N_2$
	and taking closure we get from Corollary \ref{linN2} that $[U,[A,A]]\subseteq U$. That is, $U$ is a closed Lie ideal of $[A,A]$. By Theorem \ref{LieAA}, $U$ is a Lie ideal of $A$.
\end{proof}

Let $f(x_1,\dots,x_n)$ be a polynomial in noncommuting variables with coefficients in $\C$.
Let us denote by $f(A,\dots,A)$, or $f(A)$ for short, the range of $f$ on $A$. (If $A$ is non-unital we assume that $f$ has no independent term.) Since the set  $f(A)$ is invariant by similarity, $\overline{\Lin(f(A))}$ is  a Lie ideal. It is shown in \cite[Theorem 2.3]{bresar-klep1} that even
$\Lin(f(A))$ is Lie ideal.

In the sequel by a polynomial we always understand a polynomial in noncommuting variables with coefficients in $\C$.

Recall that for each $k\in \N$ we let  $I_k$ denote the intersection of  the kernels of all representations of $A$ of dimension at most $k$.  
In the following theorem we use the conventions $I_0=A$ and $M_0(\C)=\{0\}$. We regard every polynomial as an identity on $M_0(\C)$. By a nonconstant polynomial we mean one  with positive degree  in at least one of its variables.

\begin{theorem}\label{polylie}
	Let $f$ be a nonconstant polynomial. Suppose that  $f(A)\subseteq \overline{[A,A]}$. Then
	$\overline{\Lin  (f(A))} =\overline{[I_k,A]}$, where  $k\geq 0$ is the largest number such that $f$ is an identity on $M_{k}(\C)$ (such a number must exist since no polynomial is an identity on all matrix algebras).
\end{theorem}	
\begin{proof}
	Let $I=\Id([f(A),A])$. Then $A/I$ is a subhomogeneous $\Cstar$-algebra, since it satisfies the (nontrivial) polynomial identity $[f(x_1,\dots,x_n),y]$ (see \cite[Proposition IV.1.4.6]{blackadar}). The range of $f$ on $A/I$ is both in the center of $A/I$ and  in
	$\overline{[A/I,A/I]}$, as $f(A)\subseteq \overline{[A,A]}$. But in a subhomogeneous $\Cstar$-algebra the center and the closure of the span of the commutators have zero intersection (since this is true in every finite dimensional representation). Hence, $f(A/I)=\{0\}$; i.e., $f(A)\subseteq I$. Thus, 
	$\Id(f(A))=I=\Id([f(A),A])$. By assumption, we also have that  $f(A)\subseteq \overline{[A,A]}$. It follows that  $\overline{\Lin (f(A))}=\overline{[I,A]}$ by Lemma \ref{basiclemma}. 
	
	Let us now show that $I=I_k$, with $k\geq 0$ as in the statement of the theorem.
	Let $\pi\colon A\to M_l(\C)$ be a representation of $A$ with $l\leq k$. By assumption, $f(M_l(\C))=\{0\}$. Hence, $f(A)\subseteq \ker \pi$, and so $I=\Id(f(A))\subseteq \ker \pi$. Since, by definition,  $I_k$ is the intersection of the kernels of all such $\pi$, we get that $I\subseteq I_k$. To prove the opposite inclusion notice first that  $A/I$ must be a $k$-subhomogeneous  $\Cstar$-algebra. For suppose that there exists an irreducible representation $\pi\colon A/I\to M_{m}(\C)$,  with  $m>k$.
	Since  $f$ is an identity  on   $A/I$ and $\pi$ is onto,  we get that $f$ is an identity on $M_m(\C)$. This contradicts  our choice of $k$. Hence, every irreducible representation of $A/I$ has dimension at most $k$; i.e., $A/I$ is $k$-subhomogeneous.
	Since   $I_k$ may be alternatively described as the smallest closed two-sided ideal the quotient by which is $k$-subhomogeneous, $I_k\subseteq I$.
\end{proof}

Let $s_k$ denote the standard polynomial in $k$ noncommuting variables. That is,
\[
s_k(x_1,\dots,x_k)=\sum_{\sigma\in S_k} \mathrm{sign}(\sigma)x_{\sigma(1)}\cdots x_{\sigma(k)},
\] 
where $S_k$ denotes the symmetric group on $k$ elements. The Amitsur-Levitzky theorem states that   $s_{2k}$ is a polynomial identity of minimal degree on $M_k(\C)$
(\cite{amitsur-levitzky}).
Define $\pi_1(x,y)=[x,y]$ and 
\[
\pi_{k+1}(x_1,\dots,x_{2^{k+1}})=
[\pi_k(x_1,\dots,x_{2^k}),\pi_k(x_{2^k+1},\dots,x_{2^{k+1}})]
\]
for all $k\geq 1$. The following two special cases of the previous theorem are worth remarking upon:

\begin{corollary}
	$\overline{\Lin (\sigma_{2k}(A))}=\overline{[I_k,A]}$	and $\overline{\Lin(\pi_k(A))}=\overline{[A,A]}$ for all $k\geq 1$.
\end{corollary}	
\begin{proof}
	Let $k\in \N$. It is well known that  $s_{2k}$ is expressible as a sum of commutators in the algebra of polynomials in $2k$ noncommuting variables. Hence, $s_{2k}(A)\subseteq [A,A]$. We can thus apply  Theorem \ref{polylie} to   $s_{2k}$. By the Amitsur-Levitsky theorem,  $s_{2k}$ is a polynomial identity of $M_k(\C)$ but not of $M_{k+1}(\C)$. Thus, by Theorem \ref{polylie},     $\overline{\Lin (\sigma_{2k}(A))}=\overline{[I_k,A]}$.
	
	The polynomial  $\pi_k$ is an identity on $\C$ but not on $M_2(\C)$. (In fact, by \cite[Theorem 2]{herstein2}, if $\pi_k$ is a polynomial identity on a semiprime ring without 2-torsion then the ring
	must be commutative.)	Thus, by Theorem \ref{polylie},  $\overline{\Lin(\pi_k(A))}=\overline{[A,A]}$.
\end{proof}

Let's  now give a characterization of the polynomials whose range is contained in $\overline{[A,A]}$.
Following \cite{bresar-klep1},  we say that two polynomials $f$ and $g$ (in noncommuting variables, with coefficients in $\C$) are cyclically equivalent if $f-g$
is a sum of commutators in the ring $\C(X_1,X_2,\dots)$ of polynomials in noncommuting variables.
If a polynomial is cyclically equivalent to $0$ then its range is clearly in $\overline{[A,A]}$. On the other hand, if $A$ has no bounded traces then $A=\overline{[A,A]}$ (see \cite{cuntz-pedersen}) and so any polynomial has range in $\overline{[A,A]}$. The general case is a mixture of these two.  In the following theorem  we maintain  the conventions that $I_0=A$ ,  $M_0(\C)=\{0\}$, and that every polynomial is an identity on $M_0(\C)$.

\begin{theorem}
	Let $k\geq 0$ be the smallest number such that the closed two-sided ideal $I_k$ has no bounded traces
	(set $k=\infty$ if this is never the case). Let $f$ be a nonconstant polynomial.
	\begin{enumerate}[(i)]
		\item	 
		If $k=\infty$ then
		$f(A)\subseteq \overline{[A,A]}$  
		if and only if $f$ cyclically equivalent to 0. 
		
		\item
		If $k<\infty$ then     $f(A)\subseteq \overline{[A,A]}$ 
		if and only if $f$ is  cyclically equivalent to a polynomial identity on $M_k(\C)$. 
	\end{enumerate} 
\end{theorem}	

\begin{proof}
	Let us first prove the forward implications. If $f$ is cyclically equivalent to $0$
	then clearly $f(A)\subseteq \overline{[A,A]}$. Suppose that $k<\infty$
	and that $f$ is cyclically equivalent to a polynomial $g$ which is an identity on $M_k(\C)$.
	Then $g(A)\subseteq I_k$ and $I_k=\overline{[I_k,I_k]}$, since $I_k$ has no bounded traces. Thus, $g(A)\subseteq \overline{[A,A]}$. But 
	$(f-g)(A)\subseteq [A,A]$. Thus, $f(A)\subseteq \overline{[A,A]}$, as desired.
	
	Let us suppose now that $f(A)\subseteq \overline{[A,A]}$.	 We will follow closely the proof of \cite[Theorem 4.5]{bresar-klep1} where the   result is obtained  for
	the range of polynomials on matrix algebras. If the independent term of $f$ is  nonzero then 
	$1\in f(A)\subseteq \overline{[A,A]}$. Hence, $A$ has no bounded traces; i.e., $k=0$. Since, by convention,
	any  polynomial is an identity on $M_0(\C)$, we are done. Let us assume now that $f$
	has no independent term. 
	Let $f=\sum_{i=1}^m f_i$ be the decomposition of $f$ into multihomogeneous polynomials. Then, by the proof of \cite[Theorem 2.3]{bresar-klep1},
	$f_i(A)\subseteq \Lin(f(A))$ for all $i$. This reduces the proof to the case that
	$f$ is  multihomogeneous.
	We   prove the theorem for multihomogeneous polynomials by induction on the smallest degree
	of  its variables. Suppose that the degree of $f$ on $x_1$
	is 1. Then $f$ is cyclically equivalent to a polynomial of the form $x_1g(x_2,\dots,x_n)$. Hence  $Ag(A)\subseteq \overline{[A,A]}$, which in turn implies that $\Id(g(A))\subseteq \overline{[A,A]}$.  If $g$ is 0, then $f$ is cyclically equivalent to 0 and we are done. If $g$ is constant and nonzero, then $A=\Id(g(A))\subseteq \overline{[A,A]}$. That is,  $A=\overline{[A,A]}$, $k=0$, and $f$ is an identity on $M_0(\C)$; again we are done. If $g$ is nonconstant then  $\Id(g(A))=I_{k'}$ for some $k'$ and furthermore  $g$ is an identity on $M_{k'}(\C)$ (see the proof of Theorem \ref{polylie}).
	From $I_{k'}\subseteq \overline{[A,A]}$ and Lemma \ref{bunce} we deduce that $I_{k'}=\overline{[I_{k'},I_{k'}]}$. Hence,  $I_{k'}$ has no bounded traces; i.e.,   $k'\geq k$.  It follows that $g$ is an identity on $M_{k}(\C)$, and since  
	$f=x_1g$, so is $f$. This completes the first step of the induction.

	Suppose now that $f(x_1,\dots,x_n)$ is a multihomogeneous polynomial whose variable of smallest degree, $x_n$,  has degree  $d$, with $d>1$. Consider the polynomial
	\begin{multline}
	g(x_1,\dots,x_n,x_{n+1})=f(x_1,\dots,x_{n-1},x_n+x_{n+1})\\-
	f(x_1,\dots,x_{n-1},x_{n+1})
	-f(x_1,\dots,x_n).
	\end{multline}
	Then $g(A)\subseteq \overline{[A,A]}$ and the degree of $g$   on $x_n$ is less than $d$. By induction, $g$ is cyclically equivalent to a polynomial identity on $M_k(\C)$ (if $k<\infty$) or cyclically equivalent to $0$ (if $k=\infty$).
	Since $f(x_1,\dots,x_n)=\frac{1}{2^d-2} g(x_1,\dots,x_n,x_n)$, the same holds for $f$. 
\end{proof}

\section{Finite sums and sums of products}
Recall the following basic fact: a dense two-sided ideal in a unital $\Cstar$-algebra must agree with the whole $\Cstar$-algebra (because it would intersect the ball of radius one centered at the unit, all whose elements are invertible).  It follows that if $A$ is unital and $A=\Id(X)$ then $A=AXA$. Here we exploit this fact to obtain quantitative versions of some of the results from the previous sections. 

\begin{theorem}\label{fullL}
	Let $A$ be unital and let $L$ be a Lie ideal of $A$ such that $\Id([L,A])=A$. Suppose that $L$ is linearly spanned by a  set $\Gamma\subseteq A$; i.e.,  $L=\Lin(\Gamma)$.
	Suppose furthermore that there exists $M\in \N$ such that for all $l\in \Gamma$ and $z\in A$
	the commutator $[l,z]$ is a linear combination of at most $M$ elements of  the set  $\Gamma$. The following are true:
	\begin{enumerate}[(i)]
		\item
		There exists $N$ such that every element of $A$ is expressible as a linear combination of  $N$ elements of $\Gamma$
		and $N$ products of two elements of $\Gamma$.
		
		\item
		There exists $K$ such that every single commutator $[x,y]$ in $A$ is expressible as a linear combination of $K$ elements of
		$\Gamma$.
	\end{enumerate}
\end{theorem}
\begin{proof}
	We have shown in the proof of Theorem \ref{bresar} (i)   that $\Id([L,A])=\Id([L,L])$.
	(Indeed, after setting $I=\Id([[L,A],[L,A]])$, we proceeded to show that $[L,A]\subseteq I$, which implies that
	$\Id([L,A])\subseteq I\subseteq \Id([L,L])$. Clearly, these inclusions must be equalities.) Therefore, $A=\Id([L,L])=\Id([\Gamma,\Gamma])$. Since $A$ is unital, it is algebraically generated as a two-sided ideal by   $[\Gamma,\Gamma]$. 
	Hence, 
	\[
	1=\sum_{i=1}^n x_i[k_i,l_i]y_i,
	\] 
	for some $x_i,y_i\in A$ and $k_i,l_i\in \Gamma$. Let $a\in A$. Then 
	\[
	a=\sum_{i=1}^n (ax_i)[k_i,l_i]y_i.
	\] 
	It suffices to show that each term of the sum on the right is a linear combination of a fixed number of 
	elements of $\Gamma$ and of products of two elements of $\Gamma$. We have
	the following identity (derived from the arguments in the proof of Lemma \ref{elementary} (i)):
	\[
	x[l,m]y=[xyl,m]-[xy,m]l+[xm,[y,l]]-[x,[y,l]]m+[xl,[m,y]]-[x,[m,y]]l,
	\]
	for all $x,y\in A$ and $l,m\in \Gamma$. Observe that each of the terms on the right side are of either one of the following forms: 
	$[z,l]$,  $[z,l]l'$, $[z,[z',l]]$, or $[z,[z',l]]l'$, where $z,z'\in A$ and $l,l'\in \Gamma$.
	Recall now that, by assumption, the commutators 
	$[z,l]$, with $z\in A$ and $l\in \Gamma$, are expressible as linear combinations of at most $M$ elements
	of $\Gamma$. This implies that elements of either one of the  forms mentioned before are linear combinations of either $M$
	or $M^2$ elements of $\Gamma$ or products of two elements of $\Gamma$.
	
	(ii) Let $x\in A$. By (i), $x=\sum_{i=1}^N\lambda_i l_i+\sum_{i=1}^N\mu_i m_in_i$ for some scalars
	$\lambda_i,\mu_i$ and some $l_i,m_i,n_i\in \Gamma$. Let $y\in A$. Then,
	\[
	[x,y]=\sum_{i=1}^N\lambda_i[l_i,y]+\sum_{i=1}^N\mu_i[m_i,n_iy]+\sum_{i=1}^N\mu_i[n_i,ym_i].
	\]
	Appealing to the fact that every commutator of the form $[l,z]$, with $l\in \Gamma$ and $z\in A$
	is a linear combination of at most $M$ elements of $\Gamma$, we deduce that the right side is a linear 
	combination of $3MN$ elements of $\Gamma$.
\end{proof}

\begin{theorem}\label{sumofcomms}
	Let $A$ be unital and without 1-dimensional representations. Then there exists $N\in \N$ such that every element of $A$ is expressible as a sum of the form
	\begin{equation*}
	\sum_{i=1}^N[a_i,b_i]+\sum_{i=1}^N[c_i,d_i]\cdot [c_i',d_i'].
	\end{equation*}
\end{theorem}
\begin{proof}
	The quotient $A/\Id([A,A])$ is a  commutative $\Cstar$-algebra.  If it were nonzero,  it would have  non-trivial 1-dimensional representations. But we have asssumed that $A$ has no 1-dimensional representations, Thus,  $A=\Id([A,A])$. The previous theorem is then applicable to $L=[A,A]$ and
	$\Gamma=\{[x,y]\mid x,y\in A\}$, yielding the desired result.
\end{proof}

We can link the constant  $N$  in Theorem \ref{sumofcomms}  
to a certain notion of  ``divisibility" studied in \cite{RR}.  A unital 
$\Cstar$-algebra $A$  is called weakly $(2,N)$-divisible if there exist $x_1,\dots,x_N\in N_2$ and 
$d_1,\dots,d_N\in A$ such that
\[
1=\sum_{i=1}^N d_i^*x_i^*x_i d_i.
\] 
(The definition of weakly $(2,N)$-divisible in \cite{RR} is  in terms of the Cuntz semigroup of $A$ but can be seen to be equivalent to this one.) A  unital $\Cstar$-algebra without 1-dimensional representations must be weakly $(2,N)$-divisible for some $N$ (\cite[Corollary 5.4]{RR}). This fact, combined with the following proposition,  gives  another proof of Theorem \ref{sumofcomms}. 

\begin{proposition}
	If $A$ is   unital and  weakly $(2,N)$-divisible then every element of $A$   is expressible as a sum of the form $\sum_{i=1}^N[a_i,b_i]+\sum_{i=1}^N[c_i,d_i]\cdot [c_i',d_i']$.
\end{proposition}

\begin{proof}
	Suppose that $1=\sum_{i=1}^N d_i^* x_i^*x_i d_i$, with $x_i\in N_2$ for all $i$. Let $a\in A$.
	Then 
	\[
	a=\Big(\sum_{i=1}^N d_i^* x_i^*x_i d_i\Big)\cdot a=\sum_{i=1}^N [d_i^*x_i^*,x_id_ia] + \sum_{i=1}^N x_i d_iad_i^*x_i^*.
	\]
	It thus suffices to show that $xbx^*$ is a product of 2 commutators for all $x\in N_2$ and $b\in A$.
	Say $x=v|x|$ is the polar decomposition of $x$ in $A^{**}$. Then $xbx^*=(xb|x|^{\frac 1 2})\cdot |x|^{\frac 1 2}v^*$. But both     $xb|x|^{\frac 1 2}$ and $|x|^{\frac 1 2}v^*$ belong to $N_2$. (Let us prove this for the latter: We have
	$|x|^{\frac 1 2}\in \Cstar(x^*x)\subseteq \overline{|x|Ax}$. Multipliying by $v$ on the left we get that $v|x|^{\frac 1 2}\in \overline{xAx}$. Since $x$ is a square zero element,  we deduce that  $v|x|^{\frac 1 2}$, and its adjoint, are square zero elements as well.) By  Lemma \ref{aluthge}, both $xb|x|^{\frac 1 2}$ and $|x|^{\frac 1 2}v^*$ are commutators.
\end{proof}

\begin{remark}	
	If $1\in B\subseteq A$ and $B$ is weakly $(2,N)$-divisible then so is $A$.  This observation can be used to find upper bounds on $N$ for specific examples  (e.g., when $B$ is a dimension drop $\Cstar$-algebra; see \cite[Example 3.12]{RR}).
\end{remark}

Let $P\subseteq A$ denote the set of projections of $A$.
Let us apply Theorem \ref{fullL} to $\Lin(P)$. To see that this is a Lie ideal, recall that the linear  span of the idempotents is  Lie ideal and that, by a theorem of 
Davidson (see paragraph after \cite[Theorem 4.2]{marcoux-murphy}), every idempotent is a linear combination of five projections. In Davidson's theorem, the number of projections can be reduced to four:
\begin{lemma}\label{davidson}
	Every idempotent of $A$ is a linear combination of four projections
\end{lemma}
\begin{proof}
	Let  $e\in A$ be an idempotent and let $p\in A$ denote its range projection. Then $e=p+x$, with $x\in pA(1-p)$. Let us show that $x$ is a linear combination
	of three projections. It suffices to assume that $\|x\|<\frac  1 2$. 
	For each $x\in pA(1-p)$ such that  $\|x\|<\frac  1 2$  let us define
	\[
	q(x)=
	\begin{pmatrix}
	\frac{1+\sqrt{1-4xx^*}}{2} & x\\
	x^* & \frac{1-\sqrt{1-4xx^*}}{2}
	\end{pmatrix}
	\in 
	\begin{pmatrix}
	pAp & pA(1-p)\\
	(1-p)Ap & (1-p)A(1-p)
	\end{pmatrix}.
	\] 
	A straightforward computation shows that   $q(x)$ is a projection and, furthermore, that
	\[
	x=\frac{1+i}{4}q(x)+\frac{-1+i}{4}q(-x)-\frac{i}{2}q(ix). \qedhere
	\]
\end{proof}

\begin{theorem}
	Suppose that the $\Cstar$-algebra  $A$ is  unital and  that $\Id([P,A])=A$. The following are true:
	\begin{enumerate}[(i)]
		\item
		There exists   $N$  such that every element of $A$
		is expressible as a linear combination of $N$ projections and $N$  products of two projections. 
		
		\item
		There exists $K$ such that every commutator $[x,y]$, with $x,y\in A$, is expressible as a linear combination of $K$ projections.		
	\end{enumerate}
\end{theorem}	
\begin{proof}
	Both (i) and (ii) will follow once we show that Theorem \ref{fullL} is applicable to the Lie ideal  $\Lin(P)$ and the generating set $P$.
	It suffices to show that a commutator of the form $[p,z]$, with $p$ a projection,  is a linear combination of projections with a uniform bound on the number of terms. But   
	\[
	[p,z]=(p+pz(1-p))-(p+(1-p)zp),
	\] 
	where $p+pz(1-p)$ and $p+(1-p)zp$ are idempotents. Each of them is a linear combination of four projections by Lemma \ref{davidson}. 
\end{proof}

\begin{remark} If $B$ is a unital $\Cstar$-subalgebra of $A$ and $\Id([P_B,B])=B$,
	then 
	\[
	1=\sum_{i=1}^n x_i[p_i,q_i]z_i,
	\] 
	for $x_i,y_i,z_i\in B$ and projections $p_i,q_i\in P_B$.
	It follows that the constants $N$ and $K$ that one finds for $B$ following the proof of Theorem \ref{fullL}
	applied to $L=\Lin(P_B)$ also work for the $\Cstar$-algebra $A$.
	This observation can be used to obtain concrete estimates of these constants in  cases where $B$ is rather simple.
\end{remark}

An element of  a $\Cstar$-algebra is called full if it generates the $\Cstar$-algebra as a closed two-sided ideal.  Recall also that a unital $\Cstar$-algebra is said to have real rank zero if its invertible selfadjoint elements are dense in the
set of selfadjoint elements. By \cite[Theorem V.7.3]{davidson}, this is equivalent to asking that every hereditary $\Cstar$-subalgebra of $A$ has an approximate unit consisting of projections. 
\begin{corollary}
	Suppose that $A$ is unital and either contains two full orthogonal projections or has real rank zero and no 1-dimensional representations. Then there exist $N$ and $K$ such that (i) and  (ii)  of the previous theorem hold for $A$. 
\end{corollary}

\begin{proof}
	Let us  show    in both cases that $\Id([P,A])=A$. 
	
	Say $p$ is a projection such that $p$ and $1-p$ are full; i.e, $A=\Id(p)=\Id(1-p)$. Then
	\[
	A=\Id(p)\cdot \Id(1-p)=\overline{ApA(1-p)A}=\Id(pA(1-p)).
	\] 
	On the other hand, $\Id(pA(1-p))=\Id([p,A])$. 	
	Indeed, 
	\[pA(1-p)=[p,A](1-p)\subseteq \Id([p,A]),\]
	and  conversely
	\[
	[p,A]=\{pa(1-p)-(1-p)ap\mid a\in A\}\subseteq \Id(pA(1-p)).
	\] 
	(We have $(1-p)ap\in \Id(pA(1-p))$ since closed two-sided ideals are selfadjoint.)  Hence, $A=\Id(pA(1-p))=\Id([p,A])$, as desired.
	
	Suppose now that $A$ has real rank zero and no 1-dimensional representations, i.e.,  $A=\Id([A,A])$. Since $\Id([A,A])=\Id(N_2)$ (where, as before, $N_2$ denotes the  set of nilpotents of order two),   $A=\Id(N_2)$. Furthermore, since $A$ is unital  there exist   $x_1,\dots,x_n\in N_2$ such that $A=\Id(x_1,\dots,x_n)$, for it suffices to choose these elements such that $\sum_{i=1}^n a_ix_ib_i$ is invertible for some 
	$a_i,b_i\in A$. Since $A$ has real rank-zero, the hereditary subalgebras $\overline{x_i^*Ax_i}$ have approximate units consisting of projections for all $i$. Using this, we can can find projections $p_i\in \overline{x_i^*Ax_i}$ for $i=1,\dots,n$ 
	such that $A=\Id(p_1,\dots,p_n)$.  We claim that $p_i$ 
	is Murray-von Neumann subequivalent to $1-p_i$ for all $i$. To prove this, let  $x_i=v_i|x_i|$ be the polar decomposition of $x_i$ in $A^{**}$.   Since  $p_i\in \overline{x_i^*Ax_i}$  we have $p_i\leq v_i^*v_i$. On the other hand, $x_i^2=0$ implies that $v_i^*v_i$
	and $v_iv_i^*$ are orthogonal projections. Hence, $v_iv_i^*\leq 1-p_i$. It follows that
	$p_i=(p_iv_i^*)(v_ip_i)$ and $(v_ip_i)(p_iv_i^*)=v_ip_iv_i^*\leq v_iv_i^*\leq 1-p_i$.
	This proves the claim.  We now have that  $\Id(p_i)\subseteq \Id(1-p_i)$
	for all $i=1,\dots,n$. Hence,
	\[
	\Id(p_i)=\Id(p_i)\cdot \Id(1-p_i)=\Id(p_iA(1-p_i))=\Id([p_i,A]),
	\]
	for all $i=1,\dots,n$. 
	So $A=\Id(p_1,\dots,p_n)=\Id([p_1,A],\dots,[p_n,A])$, as desired.
\end{proof}

Next we turn to the Lie ideals generated by polynomials already investigated in the previous section. As  before, by a polynomial we understand a polynomial in noncommuting variables with coefficients in $\C$.

\begin{theorem}
	Let $k\in \N$. Suppose that the $\Cstar$-algebra $A$ is unital and has no representations of dimension less than  or equal to $k$. 
	Let $f$ be a nonconstant polynomial such that  $f(A)\subseteq \overline{[A,A]}$ and which is not a polynomial identity on $M_k(\C)$.  The following are true:
	\begin{enumerate}[(i)]
		\item
		There exists  $N$ such that each element of $A$ is expressible as a linear combination of $N$ values of $f$ on $A$ and $N$ products of two values of $f$ on $A$.
		\item
		There exists $K$ such that each commutator $[x,y]$ in $A$ is expressible as a linear combination of $K$ values of $f$ on $A$.
	\end{enumerate}
\end{theorem}

\begin{proof}
	Both (i) and (ii) will follow from Theorem \ref{fullL} applied to the Lie ideal $\Lin(f(A))$, with generating set  $f(A)$, once we show the the hypotheses of that theorem are valid in this case. 
	
	Since all representations of $A$ have dimension at least $k+1$, we have $A=I_k$, where $I_k$ is as defined in the previous section. Also, by the proof of Theorem \ref{polylie},  $\Id(f(A))=I_{k'}$, where $k'$ is the largest  number such that $f$ is an identity on $M_{k'}(\C)$.  But $f$ is not an identity on $M_k(\C)$, so
	we must have that $k'\leq k$. Hence $\Id(f(A))=A$. Furthermore,  as argued in the proof of Theorem \ref{polylie},   $\Id([f(A),A])=\Id(f(A))$. Thus, $A=\Id([f(A),A])$. 
	
	To complete the proof, it remains to show  that there is a uniform bound on the number of terms expressing a commutator $[f(\overline a),y]$ as a linear combination  of elements of  $f(A)$. This is indeed true, and can be derived from the proof of \cite[Theorem 2.3]{bresar-klep1} (showing that $\Lin(f(A))$ is a Lie ideal). 
	We only sketch the argument here: Say $f=\sum_{i=1}^m f_i$
	is the decomposition of $f$ into a sum of multihomogeneous polynomials. Then, as argued in the proof of \cite[Theorem 2.3]{bresar-klep1}, relying on \cite[Lemma 2.2]{bresar-klep1}, each evaluation  $f_i(\overline a)$
	is expressible as a linear combination of at most $(d+1)^n$ values of $f$. Here $d$ is the maximum of the degrees of $f$ on its variables  and $n$ the number of variables. It thus suffices to prove the desired result  for each $f_i$, or put differently, to assume that $f$ is multihomogeneous. If $f$ is a constant polynomial then $[f(\overline a),y]=0$ and the desired conclusion holds trivially. Let us assume that $f$ is multihomogeneous and has nonzero degree. We can furthermore reduce ourselves to the multilinear case. For suppose that $f$ has degree $d>1$ on $x_n$.  Let
	\begin{multline}
	g(x_1,\dots,x_n,x_{n+1})=f(x_1,\dots,x_{n-1},x_n+x_{n+1})\\
	-f(x_1,\dots,x_{n-1},x_{n+1})-f(x_1,\dots,x_n).
	\end{multline}
	Then  the degree of $g$  on $x_n$ is less than $d$ and $f(x_1,\dots,x_n)=\frac{1}{2^d-2} g(x_1,\dots,x_n,x_n)$.
	This reduces the proof to $g$. Continuing in this way, we arrive at a multilinear polynomial.  Finally, if $f$ is multilinear  then the identity
	\[
	[f(a_1,\dots, a_n),y]=f([a_1,y],\dots,a_n)+f(a_1,[a_2,y],\dots,a_n)+\cdots+f(a_1,\dots,[a_n,y])
	\]
	shows that there is a uniform bound on the number  of terms expressing $[f(\overline a),y]$ as a linear combination
	of values of $f$. 
\end{proof}

A theorem of 
Pop (\cite[Theorem 1]{pop}) says that if $A$ is unital and without bounded traces then there exists $M\in \N$ such every element of $A$ is a sum
of $M$ commutators. Combining this with the previous theorem yields the following corollary: 
\begin{corollary}\label{polypop}
	Let $A$ be unital and without bounded traces. Let $f$ be a nonconstant polynomial. 
	Then there exists $N\in \N$ such that each element of $A$ is expressible as a linear combination of $N$ values of $f$ on $A$.
\end{corollary}
\begin{proof}
	Since $A$ has no bounded traces  it has no  finite dimensional representations. Hence $I_k=A$ for all $k\in \N$. Furthermore, $f(A)\subseteq A=\overline{[A,A]}$   by Pop's theorem.  Thus, by the preceding theorem,  every commutator is a linear combination of $K$ values  of $f$. On the other hand,  every element of $A$ is a sum of $M$ commutators (by Pop's theorem).  So every element of $A$ is a linear combination of $KM$ values of $f$.
\end{proof}

In \cite{bresar-klep2}, Bre\v{s}ar and Klep reach the conclusion of the preceding corollary for $K(H)$ and $B(H)$
(the compact and bounded operators on a Hilbert space) and   for certain rings obtained as tensor products.

Next we construct   examples showing that if $f(\C)=\{0\}$ then the number $N$ in Corollary \ref{polypop} can be arbitrarily large. Taking $f(x,y)=[x,y]$ this shows that in Pop's theorem the number  of commutators  can be arbitrarily large.
Taking $f(x_1,\dots,x_6)=[x_1,x_2]+[x_3,x_4]\cdot [x_5,x_6]$ this  shows that the $N$ in Theorem \ref{sumofcomms} can be arbitrarily large as well.

\begin{example}\label{counterpop}
	Let $f$ be a polynomial in $n$ noncommuting variables such that
	$f(\C)=\{0\}$.  Let $K\in \N$. We will construct a $\Cstar$-algebra $A$, unital and without bounded traces, and an element $e\in A$ not expressible as a linear combination of $K$ values of $f$. 
	Let $S^2$ denote the 2-dimensional sphere. Let $\eta\in M_2(C(S^2))$ be a rank one non-trivial projection
	(i.e, one not Murray-von Neumann equivalent to a constant rank one projection). 
	Choose $N\geq 2Kn$.  Let $\eta_N=\eta^{\otimes N}\in M_{2^N}(C((S^2)^N))$. It is well know that the vector bundle associated to $\eta_N^{\otimes N}$ has non-trivial Euler class. In particular, any $N$
	sections of the vector bundle associated to $\eta_N$ have a common vanishing point.

	Let $X=\prod_{i=1}^\infty (S^2)^N$. Let $1_X$ denote the unit of $C(X)$. Let $e$, $p$, and $q$ be projections  in $C(X,B(\ell^2(\N)))$  defined
	as follows:
	\begin{align*}
	e &=\mathrm{diag}(1_X,0,0,\dots),\\
	q(x_1,x_2,\dots)&=\mathrm{diag}(0,\eta_N(x_1),\eta_N(x_2),\dots),\\
	p(x_1,x_2,\dots) &=\mathrm{diag}(1_X,\eta_N(x_1),\eta_N(x_2),\dots),
	\end{align*}
	where $x_i\in (S^2)^N$ for all $i=1,2,\dots$.  The following facts are known
	(see \cite[Th\'{e}or\`{e}me 6]{dixmier-douady} and \cite[Section 4]{rordamP}): 
	\begin{enumerate}
		\item
		$q^{\oplus N+1}$ is a properly infinite projection (i.e., $q^{\oplus N+2}$
		is Murray-von Neumann subequivalent to $q^{\oplus N+1}$), 
		\item
		$e$ is not Murray-von Neumann  subequivalent to $q^{\oplus N}$. Thus, for any $N$ elements of    $qC(X,B(\ell^2(\N)))e$ (i.e., ``sections'' of $q$) there exists $x\in X$ on which they all vanish.
	\end{enumerate}
	Since $p=e\oplus q$, we  have   that $p^{\oplus N+1}$
	is also a properly infinite projection.  
	Let us define  $A=pC(X,B(\ell^2(\N)))p$.  
	Notice first that $A$ cannot have bounded traces, since its unit is stably properly infinite.
	Let us show that
	$e\in A$ cannot be approximated within a distance less than one 
	by a linear combination of $K$ elements of $f(A)$.
	Suppose, for the sake of contradiction,  that 
	\[
	\Big\|e-\sum_{i=1}^K\lambda_i f(\overline a_i)\Big\|<1.
	\]
	Multiplying by $e$ on the left and on the right we get
	\begin{equation}\label{withinlessthan1}
	\Big\|e-\sum_{i=1}^K\lambda_i ef(\overline a_i)e\Big\|<1.
	\end{equation}
	Say $\overline a_i=(a_{i,1},\dots,a_{i,n})$ for $i=1,\dots K$. 
	Since $p=e\oplus q$, we may regard  each $a_{i,j}\in A$ as an ``$e\times q$''
	matrix:
	\[
	a_{i,j}=\begin{pmatrix}
	b_{i,j} & c_{i,j} \\
	d_{i,j} & e_{i,j},
	\end{pmatrix}\in 
	\begin{pmatrix}
	eC(X,B(\ell^2))e & eC(X,B(\ell^2))q \\
	qC(X,B(\ell^2))e & qC(X,B(\ell^2))q
	\end{pmatrix}
	\]
	for all $i=1,\dots,K$ and $j=1,\dots,n$. Since $N\geq 2nK$, there exists $x\in X$
	such that $ c_{i,j}(x)=d_{i,j}(x)=0$ for all $i,j$. But $eC(X,B(\ell^2))e\cong \C$ and
	$f(\C)=0$.
	So $ef(\overline a_i(x))e=f(\overline b_i(x))=0$ for all $i=1,\dots,K$. Evaluating at $x\in X$
	in  \eqref{withinlessthan1}  we then get $\|e(x)-0\|< 1$, which is clearly impossible. 
\end{example}

\begin{remark}The previous  example shows also that the existence of a unit cannot be dropped neither in Theorem \ref{sumofcomms} nor in  Corollary \ref{polypop}. Indeed, consider $A=\bigoplus_{N=1}^\infty A_N$, 
	with $A_N$  as in the example above. Then $A$ has no  bounded traces (whence no 1-dimensional representations) but 
	$A\neq \Lin(f(A))$ for any polynomial $f$ in noncommuting variables such that $f(\C)=0$.
\end{remark}

\section{Similarity invariance and the span of  $N_2$}\label{simspan}
Let $U\subseteq A$ be a linear subspace. In this section we investigate the equivalence between the following two properties of $U$: 
\begin{enumerate}[(i)]
	\item
	$(1+x)U(1-x)\subseteq U$ for all $x\in N_2$,
	\item
	$U$ is a Lie ideal.
\end{enumerate}
We have seen in Theorem \ref{closedinvariance} 
that if $U$ is closed then (i) and (ii) are indeed equivalent. Furthermore, the proof of (ii) $\Rightarrow$ (i) in Theorem \ref{closedinvariance} is valid for any subspace $U$ of $A$. 
Thus, we are interested in the implication (i) $\Rightarrow$ (ii) when $U$ is not necessarily  closed.
In the closed case,  the  proof of (i) $\Rightarrow$ (ii) in Theorem \ref{closedinvariance}   can be split into two steps: In the first step we showed  that $U$ is a Lie ideal of $[A,A]$. This was done as follows: (i) readily implies that 
$[U,N_2]\subseteq U$. Then using that $[A,A]\subseteq \overline{\Lin(N_2)}$ (by Corollary \ref{linN2}) and  that $U$ is closed, 
we  arrived at $[U,[A,A]]\subseteq U$.  In the second step we appealed to Theorem \ref{LieAA}, showing that a closed Lie ideal of $[A,A]$ is a Lie ideal of $A$. 

Let us first address the passage from $[U,N_2]\subseteq U$ to $[U,[A,A]]\subseteq U$ in the non-closed case.
Let  $A_+$ denote the positive elements of $A$. Let us define 
\[
N_2^c=
\{
x\in A\mid xe=fx=x\hbox{ for some }e,f\in A_+\hbox{ such that }ef=0\}.
\]
One   readily checks that $N_2^c\subseteq N_2$. Let us show that
$N_2^c$ is dense in $N_2$. 
Let $x\in N_2$. Observe that for each   $\phi\in C_0(0,1]$ we have $\phi(|x|)\phi(|x^*|)=0$,
since $\phi(|x|)\in \Cstar(x^*x)$ and $\phi(|x^*|)\in \Cstar(xx^*)$. Let us choose 
$\phi_1,\phi_2,\ldots\in C_0(0,1]$, an approximate unit of $C_0(0,1]$  such that $\phi_{n+1}\phi_n=\phi_n$ for all $n$. Then $\phi_n(|x^*|)x\phi_n(|x|)\in N_2^c$ for all $n$, since we can  set $e=\phi_{n+1}(|x|)$ and 
$f=\phi_{n+1}(|x^*|)$. Furthermore, 
$\phi_n(|x^*|)x\phi_n(|x|)\to x$. Thus, $N_2^c$ is dense in $N_2$.  

Let us define  
\[
SN_2^c=\bigcup_{x\in N_2\cup\{0\}} (1+x)N_2^c(1-x).
\]
Notice that we still have $SN_2^c\subseteq N_2$. 

\begin{lemma}\label{N2cLie}
	$\Lin(SN_2^c)$ is a Lie ideal.
\end{lemma}
\begin{proof}
	It suffices to show that 
	$[A,x]\subseteq  \Lin(SN_2^c)$ for all $x\in N_2^c$. For then, conjugating by
	the algebra automorphism $a\mapsto (1+y)a(1-y)$, with $y\in N_2$, and using the invariance of $SN_2^c$ under such automorphisms, we get that $[A,(1+y)x(1-x)]\subseteq \Lin(SN_2^c)$ for all $x\in N_2^c$ and $y\in N_2$, as desired.
	
	Let $x\in N_2^c$. Let $e$ and $f$ be positive elements such that $xe=fx=x$ and $ef=0$. Using functional calculus on $e$, let us find   positive contractions $e_0,e_1,e_2,e_3\in \Cstar(e)$ such that $e_0e_1=e_1$, $e_1e_2=e_2$, $e_2e_3=e_3$ and $xe_3=x$.
	Similarly, let  us find positive contractions $f_0,f_1,f_2,f_3\in \Cstar(f)$ such that 
	$f_0f_1=f_1$, $f_1f_2=f_2$, $f_2f_3=f_3$ and $f_3x=x$. Note that $xe_i=f_jx=x$
	and $e_if_j=0$ for all $i,j=0,1,2,3$.
	Now let $a\in A$.
	Then
	\begin{align*}
	ax-xa &=ax-e_1ax+e_1ax-xaf_1+xaf_1-xa\\
	&=(1-e_1)ax+[e_1af_1,x]-xa(1-f_1).
	\end{align*}
	The term $(1-e_1)ax$ is in $N_2^c$. Indeed, $1-e_2$ and $e_3$ act as multiplicative units on the left and on the right of $(1-e_1)ax$ and 
	$(1-e_2)e_3=0$. 
	We check similarly  that $xa(1-f_1)$ is in $N_2^c$.
	As for $[e_1af_1,x]$ (a commutator of elements in $N_2^c$), we have that
	\[
	[e_1af_1,x]=(1+e_1af_1)x(1-e_1af_1)+(e_1af_1)x(e_1af_1)-x.
	\]
	The first term on the right belongs to $SN_2^c$. The other two have multiplicative units $e_0$ and $f_0$ on the left and on the right and thus belong  to $N_2^c$. 
\end{proof}

The following theorem answers Question \ref{linN2question} affirmatively when $A$ is unital and without 1-dimensional representations.

\begin{theorem}\label{linN2comm}
	Suppose that $A$ has no 1-dimensional representations.
	Then  
	\[\Lin(SN_2^c)=[\mathrm{Ped}(A),\mathrm{Ped}(A)].
	\]	 
	If in addition  $A$
	is unital, then $\Lin(N_2)=[A,A]$. Furthermore, in the unital case there exists $K\in \N$ such that every single commutator $[x,y]$ in $A$ is a sum of at most $K$ square zero elements.  
\end{theorem}	

\begin{proof}
	Let $P=\mathrm{Ped}(A)$. Let us first show that $SN_2^c\subseteq [P,P]$. By the similarity invariance of $[P,P]$, it suffices to show that $N_2^c\subseteq [P,P]$. Let $x\in N_2^c$ and let $e,f\in A_+$ be such that $xe=x=fx$ and $ef=0$. From the description of the Pedersen ideal in \cite[Theorem 5.6.1]{pedersen} we know  that $g(e)\in P$ for any $g\in C_0(0,\infty)_+$ of compact support, and since $xg(e)=xg(1)$, we deduce that $x\in P$. Hence, $x=[x,e]\in [P,A]$. Since $[P,A]=[P,P]$ by Lemma \ref{PPPA}, $x\in [P,P]$. This shows that
	$\Lin(SN_2^c)\subseteq [P,P]$.
	Notice now that 
	\[
	\overline{[\Lin(SN_2^c),A]}=
	\overline{[\Lin(N_2),A]}=\overline{[[A,A],A]}
	=\overline{[A,A]}.
	\]
	But $[P,P]$ is the smallest Lie ideal such that $\overline{[L,A]}=\overline{[A,A]}$, by Corollary \ref{smallest}.  (To apply Corollary \ref{smallest} we have used that $\Id([A,A])=A$, since $A$ has no 1-dimensional representations.) Thus, $[P,P]\subseteq \Lin(SN_2^c)$.
	
	Let us now assume that $A$ is unital. In this case $P=A$, so  $[A,A]=\Lin(SN_2^c)$. But $\Lin(SN_2^c)\subseteq \Lin(N_2)\subseteq [A,A]$. Thus, $\Lin(N_2)=[A,A]$. 
	
	To deduce the existence of $K$ we will apply
	Theorem \ref{fullL}  to the Lie ideal  $[A,A]$,  with generating set $SN_2^c$. Notice first that $\Id([[A,A],A])=\Id([A,A])=\Id(A)$, since $A$ has no 1-dimensional representations. It remains to  show that there is a uniform bound on the  
	number of terms expressing a commutator of the  form $[x,a]$, with
	$x\in SN_2^c$ and $a\in A$,  as a linear combination of elements of $SN_2^c$. The proof of   Lemma \ref{N2cLie} shows that such commutators are sums of at most 
	five elements of $SN_2^c$. 
\end{proof}

For infinite von Neumann algebras, the following corollary is \cite[Theorem 2]{miers}. (Miers also considered closed subspaces of von Neumann algebras, which we have already dealt with in  Theorem \ref{closedinvariance}.)
\begin{corollary}\label{simneu}
	Suppose that $A$ is either unital and without bounded traces or a von Neumann algebra. Then a subspace $U$ of $A$ is a Lie ideal if and only if $(1+x)U(1-x)\subseteq U$ for all $x\in N_2$.
\end{corollary}
\begin{proof}
	That a Lie ideal satisfies the similarity invariance of the statement has already been shown in the proof of Theorem \ref{closedinvariance}. So let us suppose that 
	$U$ is a subspace  such that  $(1+x)U(1-x)\subseteq U$ for all $x\in N_2$.
	As remarked at the start of this section,  this implies that
	$[U,N_2]\subseteq U$.
	Let us consider first the case that $A$ is unital and without bounded traces.  Then $\Lin(N_2)=[A,A]$, since  $A$ is unital and has no 1-dimensional representations (since it has  no bounded traces). Thus, $[U,[A,A]]\subseteq U$. Furthermore, $[A,A]=A$, by Pop's theorem.
	Hence, $[U,A]\subseteq U$; i.e., $U$ is a Lie ideal. 
	
	Suppose now that $A$ is a von Neumann algebra.  Let us show again that 
	$\Lin(N_2)=[A,A]$ and that if $U$ is a  Lie ideal of $[A,A]$ then it is a Lie ideal of $A$. The latter is  \cite[Lemma 3]{miers} and can be proven as follows:  
	In a von Neumann algebra we have  $A=Z(A)+[A,A]$, where $Z(A)$ denotes the center of $A$ (if $A$ is infinite, because $A=[A,A]$, and if $A$ is finite, by \cite[Theorem 3.2]{fack}); so $[U,A]=[U,[A,A]]$ for any subset $U$ of $A$.  
	Let us now show that $\Lin(N_2)=[A,A]$. 
	The  ideal $\Id([A,A])$ is closed in the ultraweak topology of $A$, being the intersection of the kernels of all 1-dimensional representations of $A$ (which are always ultraweakly continuous). 
	Thus, $\Id([A,A])$
	is also a von Neumann algebra. In particular, it is unital. Since it  has no  1-dimensional representations,   
	\[
	\Lin(N_2)=[\Id([A,A]),\Id([A,A])]\supseteq [[A,A],[A,A]].
	\] 
	From $A=[A,A]+Z(A)$ we get that $[A,A]=[[A,A],[A,A]]$. Hence, $\Lin(N_2)=[A,A]$.
\end{proof}

The passage from $U$ being a Lie ideal of $[A,A]$  to being a Lie ideal of $A$ can also be made assuming that $A$ is unital and that
$[U,A]$ is  full:  

\begin{lemma}\label{fullU}
	Suppose that  $A$ is  unital. If $U$ is a  Lie ideal of $[A,A]$ such that   $\Id([U,A])=A$   then $[A,A]\subseteq U$ (so $U$ is a Lie ideal of $A$).
\end{lemma}
\begin{proof}
	Let $V=[U,U]$, $W=[V,V]$, and $X=[W,W]$. We have shown in the proof of Theorem \ref{LieAA} that  $\Id([U,A])=\Id([X,X])$. So $A=\Id([X,X])$.
	Since $A$ is unital, the set $[X,X]$ generates $A$ algebraically as a two-sided ideal. But  $A[X,X]A\subseteq [U,U]+[U,U]^2$, by Lemma \ref{4U}.
	Hence,  $A=[U,U]+[U,U]^2$. Then, 
	\[
	[A,A]=[[U,U]+[U,U]^2,A]=[[U,U],A]\subseteq [U,[U,A]]\subseteq U.\qedhere
	\] 
\end{proof}

\begin{theorem}
	Suppose that  $A$ is  unital  and without 1-dimensional representations.  Let $U$ be a subspace of $A$ 
	such that $\Id([U,A])=A$. If 
	$(1+x)U(1-x)\subseteq U$ for all $x\in N_2$ then $[A,A]\subseteq U$.
\end{theorem}
\begin{proof}
	The similarity invariance of $U$ implies that $[U,N_2]\subseteq U$ and by Theorem \ref{linN2comm} we get that $[U,[A,A]]\subseteq U$. The previous lemma then shows that $[A,A]\subseteq U$.
\end{proof}

\begin{corollary}\label{simpleamitsur}
	Let $A$ be simple and unital. A subspace $U$ of $A$ is a Lie ideal
	if and only if $(1+x)U(1-x)\subseteq U$ for all $x\in N_2$.
\end{corollary}	

\begin{proof}
	Since $A$ is simple we have either  $\Id([U,A])=0$ or $\Id([U,A])=A$.
	If $\Id([U,A])=0$ then   $U$ is a subset of the center, which by the simplicity of $A$  is $\C$. If $\Id([U,A])=A$  then by the previous theorem $[A,A]\subseteq U$. In either case it follows that $U$ is a Lie ideal of $A$.
\end{proof}

Amitsur's \cite[Theorem 1]{amitsur}  (that a similarity invariant  subspace of a simple algebra must be a Lie ideal) requires the existence of a nontrivial idempotent in the algebra. An example in \cite{amitsur} shows that this hypothesis cannot be dropped.
Corollary \ref{simpleamitsur}  shows, however, that for simple unital $\Cstar$-algebras this assumption is not necessary
(even though they may well fail to have any  nontrivial idempotents).

	\bigskip
{	\small
	
	LEONEL ROBERT, \textsc{Department of Mathematics, University of Louisiana at Lafayette, Lafayette, Louisiana 70506}
	\par\nopagebreak
	\textit{E-mail}: \texttt{lrobert@louisiana.edu}
}

\end{document}